\documentclass[review,onefignum,onetabnum]{siamonline190516}

\usepackage{lipsum}
\usepackage{amsfonts}
\usepackage{verbatim}
\usepackage{graphicx}
\usepackage{epstopdf}
\usepackage{algorithmic}
\usepackage{bbm}
\ifpdf
  \DeclareGraphicsExtensions{.eps,.pdf,.png,.jpg}
\else
  \DeclareGraphicsExtensions{.eps}
\fi

\usepackage{enumitem}
\setlist[enumerate]{leftmargin=.5in}
\setlist[itemize]{leftmargin=.5in}

\newcommand{\qoi}{\dfrac{d\langle J\rangle^s}{ds}\Big|_{s=0}}
\newcommand{\qois}{\langle J, \partial_s\mu^s|_{s=0}\rangle}

\newcommand{\T}{{\cal T}}
\newcommand{\A}{{\cal A}}

\newcommand{\Eu}{E^{\rm u}}
\newcommand{\Es}{E^{\rm s}}

\newcommand{\Xu}{X^{\rm u}}
\newcommand{\Xs}{X^{\rm s}}

\newcommand{\norm}[1]{\left\lVert #1 \right\rVert}
\newcommand{\abs}[1]{\left\lvert #1 \right\rvert}
\newcommand{\zetas}{\zeta^{\rm s}}

\newcommand{\cs}{c^{\rm s}}
\newcommand{\cu}{c^{\rm u}}

\newsiamremark{remark}{Remark}
\newsiamremark{hypothesis}{Hypothesis}
\crefname{hypothesis}{Hypothesis}{Hypotheses}
\newsiamthm{claim}{Claim}
\newsiamthm{fact}{Fact}
\newsiamthm{property}{Property}
\newsiamthm{formula}{Formula}
\newsiamremark{example}{Example}
\headers{Differentiating statistics in chaotic systems}{N. Chandramoorthy, and Q. Wang}

\title{A computable realization of Ruelle's formula for linear response 
of statistics in chaotic systems
\thanks{Submitted to the editors 2020/02/09.
\funding{This work was supported by the U.S. Air Force Office of Scientific Research Award 
FA8650-19-C-2207 under Fariba Fahroo and Jean-luc Cambier, and by 
the Office of Naval Research 
under grant no N00014-17-1-2959.}}}

\author{Nisha Chandramoorthy\thanks{Department of Mechanical Engineering, 
Massachusetts Institute of Technology, Cambridge, MA 
  (\email{nishac@mit.edu}).}
\and Qiqi Wang\thanks{Department of Aeronautics and Astronautics, Massachusetts Institute
of Technology, Cambridge, MA 
  (\email{qiqi@mit.edu}).}}

\usepackage{amsopn}

\ifpdf
\hypersetup{
  pdftitle={A computable realization of Ruelle's formula for linear 
  response of statistics in chaotic systems},
  pdfauthor={N. Chandramoorthy, and Q. Wang}
}
\fi

\begin{document}

\maketitle

\begin{abstract}
We present a computable reformulation of 
Ruelle's linear response formula for chaotic systems.
The new formula, called Space-Split Sensitivity or S3,
achieves an error convergence of the order ${\cal O}(1/\sqrt{N})$ using 
$N$ phase points. The reformulation 
is based on splitting the overall sensitivity into that to stable and unstable 
components of the perturbation. The unstable contribution to the sensitivity 
is regularized using ergodic properties and the hyperbolic structure 
of the dynamics. Numerical examples of uniformly hyperbolic attractors 
are used to validate the S3 formula against a na\"ive finite-difference 
calculation; sensitivities match closely, with far fewer sample points required 
	by S3.
\end{abstract}



\section{Introduction}
Given a dynamical model, how much do the outputs change
in response to small changes in input parameters? The computational 
aspect of this common question, which arises across science and engineering, 
entails the mature discipline of sensitivity analysis. Using 
numerical simulations or experimental 
observations of dynamical systems, 
the computed responses or sensitivities have enabled 
multidisciplinary design optimization, uncertainty quantification and 
parameter estimation in diverse non-chaotic models (see 
\cite{peter, capecelatro} for recent surveys in aerodynamic systems). 
Sensitivity analysis in chaotic systems, however, still remains nascent.
This is because traditional sensitivity computation is done through 
linear perturbation methods including tangent or adjoint equations, 
and automatic differentiation, but 
these are inherently unstable in chaotic systems \cite{angxiu, qiqi}. 
More sophisticated techniques are needed to calculate 
the long-term effects of parameter changes in chaotic systems. 
Some of them are being actively investigated as of this writing,
and face challenges such as 
lack of convergence guarantees or prohibitive 
computational cost. In this work, we develop an alternative 
method\footnote{a preliminary version of the derivation of the method has 
appeared in \cite{nisha-aiaa}} for addressing some of these challenges. We propose 
a formula for differentiating statistical averages in 
chaotic systems to parameter perturbations, in a way that is 
amenable to computation.  
We postpone until section \ref{sec:problemStatement} the precise 
definition of the statistical response we aim to compute 
and the desired qualities of a computational solution.

Currently, the more 
sophisticated approaches to chaotic sensitivity computation include a) the ensemble 
sensitivity method b) shadowing-based approaches and c) perturbation 
methods on the transfer operator. The non-intrusive 
least squares shadowing (NILSS) method (see \cite{patrick,angxiu-nilsas},
and \cite{qiqi} for the different versions of NILSS) 
computes a shadowing perturbation that remains 
bounded in a long time window under the tangent dynamics. However, the sensitivity computed 
by using the shadowing tangent solution is 
not guaranteed to be an unbiased estimate of the true sensitivity 
\cite{patrick-KS}. 
This is because while ergodic averages converge for almost
every trajectory, there are measure zero subsets of the attractor (e.g. unstable periodic 
orbits \cite{upo}) on which they do not converge. We therefore seek an alternative that does not rely on computations 
along a single trajectory that is not guaranteed to be typical. The Lea-Allen-Haine ensemble sensitivity 
method \cite{eyink} suggests a work-around to 
the exponentially diverging sensitivities computed by the conventional tangent/adjoint methods, 
by truncating the tangent/adjoint equations at a short time, and taking a sample of average 
of many such short-time sensitivities. But, although these sample averages 
converge in the infinite time limit,
they are prohibitively expensive because the variance of tangent/adjoint solutions increases exponentially
with time (\cite{nisha_ES}).

In a recent work, Crimmins and Froyland \cite{harry} have developed a new 
Fourier analytic method for constructing the Sinai-Ruelle-Bowen or SRB measures, 
the invariant probability distributions perturbations of which we 
are interested in, of uniformly hyperbolic dynamics on tori. 
They construct a matrix representation of perturbed transfer operators that 
are quasi-compact on certain 
anisotropic Banach spaces (see \cite{gouezel} for uniformly 
 hyperbolic systems in particular, and \cite{baladi} for a recent review). Using this matrix representation, the leading eigenvector, which is the SRB measure,
 is then approximated. Although equipped with a strong theoretical basis, this 
 method would be overkill for our purpose, since we do not need to 
 construct the SRB measure, 
 but only compute the sensitivity of a given expectation
 with respect to it. Moreover, methods based on perturbations of the transfer operator such as 
\cite{harry, lucarini} typically have a computational cost that 
scales poorly with the problem dimension since they 
either involve Markov partitions (specific discretizations 
of the attractor) or need a number of basis functions to approximate 
the eigendistribution (the SRB measure) that scales with the dimension. 
Another recent method
\cite{davide} computes sensitivities by solving an adjoint equation 
as a boundary value problem on periodic orbits, although questions 
surrounding the convergence of the computed sensitivity to the 
true sensitivity (like in the shadowing methods) must be investigated further.

The strategy developed in this paper, space-split sensitivity 
or S3, deviates from that 
of all the above-mentioned methods. However, like ensemble sensitivity 
methods, it builds upon Ruelle's formula. While ensemble sensitivity suffers 
from the unbounded variance of the
unstable contribution to the overall sensitivity, S3 splits the contributions and performs a
finite-sample averaging of tangent equation solutions only for the stable
contribution. The unstable contribution is manipulated through integration-by-parts and using 
the statistical stationarity (measure preservation) of the system to yield a 
computation that does not
use unstable tangent solutions. Since both parts of the sensitivity are computed through
sampling on generic flow trajectories, the problem of the computed sensitivities corresponding
to atypical trajectories, which shadowing-based methods are vulnerable to, is averted. 
The paper is organized as follows. In the next section, we give the problem 
setting and state results from dynamical systems theory that are 
used in the derivation of S3. Our main results are stated in \cref{sec:main}, 
and the derivation of the S3 formula follows in \cref{sec:derivation}. 
An interpretation of the unstable contribution, as derived 
in S3, is presented in \cref{sec:unstableContributionInterpretation}, 
and some comments on its computational details can be found in 
\cref{sec:algorithm}. Numerical examples demonstrating a na\"ive 
implementation of the S3 formula  
are reported in \cref{sec:results}, and the conclusions follow in
\cref{sec:conclusions}.

\section{Problem statement}
\label{sec:problemStatement}
In this section, we define the output quantity of interest for S3 computation.
Where they are used, we provide a brief description of concepts 
from ergodic theory and dynamical systems, in order to make it a self-contained 
presentation for computational scientists from 
different fields. 
\subsection{The primal dynamics}
Our primal system is a $C^2$ diffeomorphism $\varphi^s$,
on a $d$-dimensional compact manifold $M$, parameterized 
by a set of parameters $s$,
\begin{align}
	\label{eqn:primal}
	u_{n+1} = \varphi^s (u_n), \: n \in \mathbb{Z}, u_n \in M.
\end{align}
In a numerical simulation, $u_n \in M$ represents a solution 
state (a $d$-dimensional vector) at time $n$ and the 
transformation $\varphi^s$ can be thought 
of as advancing by one timestep. For 
simplicity, $s$ is assumed to be a scalar. Without loss of generality, 
we take $s = 0$
to be the reference value and the map $\varphi^0$ is simply written 
as $\varphi$, without the superscript. The state vector $u_n, n\in \mathbb{Z}^+$ 
is a function of the initial state $u_0$; explicitly, $u_n = \varphi_n(u_0)$, 
where the subscript $n \in \mathbb{Z}^+$ in $\varphi_n^s$ refers to an $n$-time composition 
of $\varphi^s$, and $\varphi^s_0$ is the 
identity map on $M$. We also use the notation $\varphi_{-n}, n \in \mathbb{Z}^+$
to denote the inverse transformation $(\varphi)^{-1}$ composed with itself $n$ times; 
that is, $\varphi_{-n} (u_n) = u_0$.
\subsection{Ensemble and ergodic averages}
We assume that the dynamics $\varphi^s$ preserves an ergodic, physical probability measure $\mu^s$ of an SRB-type (see \cite{young} for an introduction to SRB 
measures), which gives us a statistical description of the dynamics. 
In particular, expectations with respect to $\mu^s$ can be observed as infinitely long time averages 
along trajectories: if $f \in L^1(\mu^s)$ is a scalar function, then $\lim_{N\to\infty}(1/N)\sum_{n=0}^{N-1}f(\varphi^s_n(u)) = 
\langle f,\mu^s\rangle, u$ Lebesgue-a.e on the basin of attraction. The superscript
$s$ in $\mu^s$ emphasizes the dependence of the SRB measure on the parameter; 
simply $\mu$, without the superscript, refers to the SRB measure at $s=0$. 
The ensemble average or the expectation of $f$ with respect to $\mu^s$, 
$\langle f, \mu^s\rangle$, is a distributional pairing 
of $f$ with $\mu^s$: the integral of $f$ on the phase space weighted 
by $\mu^s$. We sometimes use the shorter notation 
$\langle f\rangle^s$, and without the superscript at $s=0$,
to denote the same quantity. The 
infinite time average, called the ergodic average, is the more natural
form of $\langle f\rangle^s$ from the computational/experimental
standpoint, since it can be obtained by numerical evaluation/measurements 
of $f$ along trajectories. In practice the ergodic 
average is computed up to a large $N$ and this is used 
to approximate the ensemble average $\langle f\rangle^s$ -- by ergodic 
average, we mean this long but finite time average, in the remainder of 
this work. 

\subsection{Quantity of interest}
We are interested in determining the sensitivity 
of the ensemble average of an objective function $J \in 
\mathcal{C}^2(M)$, $\langle J\rangle^s$, to $s$: 
$d_s\langle J\rangle^s = \langle J, \partial_s\mu^s\rangle$. 
The regularity of $J$, along with other assumptions 
on the dynamics which we detail where they appear, 
are such that the linear response formula of Ruelle \cite{ruelle,ruelle1}
is satisfied. The objective function $J$ can also explicitly depend on $s$,
and assuming that $J$ is continuously differentiable 
with respect to $s$, the 
quantity of interest $d_s\langle J\rangle^s = \langle \partial_s 
J, \mu^s\rangle + \langle J, \partial_s\mu^s\rangle$. 
As we will see however, in a chaotic system, the mathematical or 
algorithmic difficulty lies in computing the derivative 
of the SRB measure in the second term, and not in the first, 
which can be computed as an ergodic average, 
assuming the function $\partial_s J(u,s)$ is known. 
Thus from now on we ignore the
first term and develop an algorithm 
for  $\langle J, \partial_s\mu^s\rangle$, which is nontrivial to compute -- 
we describe precisely why, in the next section.
\subsection{Ruelle's formula and its 
computational inefficiency}
\label{subsec:inefficiency}
To introduce Ruelle's formula, let the matrix-valued function $D\varphi:M\to GL(d)$\footnote{$GL(d)$ is 
the set of all invertible matrices of dimension $d\times d$.} give 
us the Jacobian matrix of the transformation $\varphi$. 
Here $D$ refers to the derivative with respect to phase space; so, the 
Jacobian at $u$, $D\varphi(u)$, is a map from the tangent space of $M$ 
at $u$, denoted as $T_u M$, to $T_{\varphi(u)} M$. We now
introduce a more succinct notation for the tangent operator 
in the following definition. 
\begin{definition}{\label{defn:tangentCocycle}}
	The tangent operator $\T(u, n):T_u M \to T_{\varphi_n(u)} M$
	is a linear operator (a matrix) for each 
	$n \in \mathbb{Z}$ and is defined as the derivative of 
	$\varphi_{n}$ with respect to the state vector $u$, 
	evaluated at $u$. By this definition, it can be written 
	as Jacobian matrix products, as follows, 
	\begin{align*} 
		\T(u, n):= \Bigg\{ \begin{array}{lr}
		D\varphi(\varphi_{n-1}(u))\cdots D\varphi(u), & n > 0 \\
			(D\varphi(\varphi_{n}(u)))^{-1}\cdots (D\varphi
			(\varphi_{-1}(u)))^{-1}, & n < 0 \\
		{\rm Id}, & n = 0.
		\end{array}
	\end{align*}
\end{definition}
From the linear response theory that was rigorously 
developed by Ruelle \cite{ruelle, ruelle1}, we have the following 
formula for the sensitivity of interest,
\begin{align}
		\qoi = \sum_{n=0}^\infty 
		\langle D(J \circ \varphi_{n}(\cdot)) \cdot 
		X(\cdot), \mu\rangle
		\label{eqn:ruelleResponseFormula},
\end{align}
where $X(u) := \partial_s \varphi^s (\varphi^{-1} (u))$ denotes the 
vector field corresponding to parameter perturbation. Here 
$DJ$ refers to the derivative of $J$ with respect to the phase 
space. For brevity, we adopt the subscript notation also 
with scalar and vector fields, as explained below: 
\begin{enumerate}
\item 
if $f:M \to \mathbb{R}$ 
is a scalar field, $f_n$ is used to denote the function $f\circ\varphi_n$.  
\item The derivative 
of $f_n = f\circ\varphi_n$, evaluated at a $\mu$-typical point $u$, is 
		denoted as $Df_n(u) := D (f\circ\varphi_n)(u)$. 
On the other hand, $(Df)_n$ refers to the 
derivative of $f$ evaluated at $u_n = \varphi_n(u)$,
where again $u$ is a $\mu$-typical point.
\item if $V$ is a vector field, then $V_n(u) \in T_{\varphi_n(u)} M$
	denotes its value at $u_n$: $V_n(u) := V(u_n)$. 
\end{enumerate}
Using this notation, the integrand in the $n$th summand of Ruelle's formula 
(Eq. \ref{eqn:ruelleResponseFormula}) can be written as $D(J 
\circ \varphi_{n}) \cdot X = DJ_n \cdot X$. Note that 
$ DJ_n \cdot X = (DJ)_n \cdot \T(\cdot, n) X$, 
is the instantaneous sensitivity of $J_n = J \circ \varphi_n$ 
to an infinitesimal perturbation to $u$ along $X(u)$. 
So, the $n$-th summand is the ensemble sensitivity of the function $J_n$. In a chaotic system, the integrand 
exhibits exponential growth $\mu$-a.e. 
The reason is that, by the definition of chaos, 
the norm of an infinitesimal perturbation to a $\mu$-a.e. 
initial condition $u$, along 
a direction $X(u)$, asymptotically grows exponentially with time, for almost 
every $X(u)$. More clearly, at $u$ $\mu$-a.e., for almost every tangent 
vector $X(u) \in T_u M$, where $\norm{\cdot}$ indicates the Euclidean norm 
in $\mathbb{R}^d$, 
$$\limsup_{n\to\infty}\dfrac{\log\norm{\T(u, n)X(u)}}{n} > 0 .$$ 

For a generic $J$, $\abs{D(J \circ \varphi_n) \cdot X}$ ($\abs{\cdot}$ 
denotes the Euclidean norm on $\mathbb{R}$) shows the 
same asymptotic trend as 
$\norm{\T(u,n)X(u)}$, that is, the integrand in Ruelle's formula, 
denoted as $L_n(u) := (DJ)_n\:\cdot \T(u,n)\:X(u)$
grows exponentially in norm, with $n$, at almost every $u$. 
This makes Ruelle's formula inefficient to evaluate directly,
because the number of samples 
required to accurately obtain the $n$th term in the series, grows rapidly with 
$n$. It is worth noting that despite the pointwise exponential growth, 
since Ruelle's formula converges, i.e., the series 
in Eq. \ref{eqn:ruelleResponseFormula} converges, the 
ensemble averages $\langle L_n \rangle$ tend to 0 as 
$n \to \infty$, due to cancellations in phase space.

The pointwise exponential 
growth of the sensitivities is however manifest in the 
variance of the direct approximation of the formula, which 
has been shown to be computationally intractable in practical 
examples \cite{nisha_ES, eyink}. Specifically, 
evaluating the series upto $N$, the variance of $\sum_{n=0}^{N} L_n$, 
grows exponentially with $N$, at almost every $u$.
As illustrated by previous numerical results \cite{nisha_ES, eyink},
this leads to the least mean squared error achievable at a given 
computational cost, to reduce poorly with the cost, usually much 
worse than in a typical 
Monte Carlo simulation. 
\section{Main contributions}
\label{sec:main}
The main contribution of this paper is a reformulation of Ruelle's formula (
Eq.\ref{eqn:ruelleResponseFormula}) into a different ensemble average 
that provably converges like a typical 
Monte Carlo computation. That is, the error in the
$N$-term ergodic average approximation of the 
alternative ensemble average asymptotically declines
as ${\cal O}(1/\sqrt{N})$. 
The reformulation relies on the following 
result, which we also prove: the 
effect of an unstable perturbation on the ensemble 
average of an objective function is captured by 
its time correlation with a certain, bounded 
function. The formal statement is as follows.
\begin{theorem}{\label{thm:rued}}
Given an unstable covariant Lyapunov vector $V^i$, $1\leq i \leq d_u$, 
there exist bounded scalar functions $g^i$ such that 
for any $J \in C^2(M)$, 
\begin{align}
	\label{eqn:theoremUnstable}
\sum_{n=0}^\infty \langle D (J \circ \varphi_n) \cdot V^i, \mu\rangle 
	= \sum_{n=0}^\infty  
	\langle J\circ\varphi_n\; g^i, \mu\rangle.
\end{align}
\end{theorem}
While the proof is delayed until \cref{sec:unstableContribution}, 
here we mention the foremost implication of the above theorem for 
computation. Since the explanation in \ref{subsec:inefficiency} applies 
to any unstable perturbation field, the left 
hand side of \ref{eqn:theoremUnstable} does not yield a Monte Carlo 
computation. On the other hand, the ergodic average approximation 
of the right hand side does. That is,
the rate of convergence would be $-1/2$, independent of the system dimension, 
for the right hand side. 
Moreover, since the right hand side enables an ergodic average approximation,
a discretization of the phase space, which is impractical in high-dimensional 
systems, is not required. 

The main result of this paper, a computable realization of Ruelle's 
formula, follows as a corollary to Theorem \ref{thm:rued}, and 
can be stated as follows:
\begin{corollary}{\label{thm:derivationOfS3}}
	In uniformly hyperbolic systems, 
	Ruelle's formula in Eq. \ref{eqn:ruelleResponseFormula} is equivalent to 
	the following sum of two exponentially converging series, 
	\begin{align}
		\label{eqn:newFormula}
		\qoi &= 
		\sum_{n=0}^\infty \langle D(J\circ\varphi_n)
		\cdot \Xs, \mu \rangle + \sum_{n=0}^\infty 
		\langle J \circ \varphi_{n}\;g, \mu\rangle,
	\end{align}
	where $g \in L^\infty(\mu)$, and $X = \Xu + \Xs$ is 
	the decomposition of $X$ along the unstable and stable 
	Oseledets spaces respectively.  
\end{corollary}
The regularized expression that is Eq. \ref{eqn:newFormula}, is referred to as 
the space-split sensitivity or S3 formula. We briefly explain how 
Corollary \ref{thm:derivationOfS3} is obtained, 
and direct the reader to \cref{sec:derivation} for 
a complete presentation. We first split Ruelle's formula (Eq. \ref{eqn:ruelleResponseFormula}) into two terms 
by decomposing the vector field $X$ into its unstable and stable components. 
The first term in Eq. \ref{eqn:newFormula} directly 
results from this splitting. As shown in \cref{subsec:stableContribution}, 
it leads to an efficient 
Monte Carlo computation and hence is left unchanged.
The other term however, which contains the contribution from the 
unstable perturbation $\Xu$, needs modification. First, it is
further split into 
$d_u$ terms by writing $\Xu$ in the basis of the unstable 
covariant Lyapunov vectors or CLVs, $V^i, 1\leq i \leq d_u$. Assuming that each of the 
$d_u$ series is well-defined, each is modified by using Theorem 
\ref{thm:rued}. 

Before we proceed with the derivation of Eq. \ref{eqn:newFormula}, 
we close this section by discussing the advantages 
offered by the new formula in Eq. \ref{eqn:newFormula}. 
The first advantage is of course the Monte Carlo convergence 
of the new formula, while no such guaranteed problem-independent 
convergence rate can be ascribed to the direct evaluation of 
the original formula in Eq. \ref{eqn:ruelleResponseFormula}. Secondly, the S3 algorithm, an efficient implementation 
of Eq. \ref{eqn:newFormula} the details of which are deferred to 
a future work, only uses information obtained 
along trajectories, and therefore does not exhibit a direct 
scaling with the problem dimension. Thirdly, since Eq. \ref{eqn:newFormula} 
is an equivalent restatement of Ruelle's formula, the 
convergence of Eq. \ref{eqn:newFormula} to the true sensitivity 
is immediate from the convergence of Ruelle's formula \cite{ruelle}.

\section{Derivation of the S3 formula}
\label{sec:derivation}
Our goal is to find an alternative representation of 
the formula \ref{eqn:ruelleResponseFormula} that can be easily computed. For this, 
we begin by splitting the parameter perturbation vector $X$ into its 
stable and unstable components, along stable and unstable Oseledets spaces, 
denoted by $\Es$ and $\Eu$ respectively. The motivation for splitting $X$ 
becomes clear when we define the subspaces. The reader is referred 
to Chapter 4 of \cite{arnold} for a detailed exposition on Oseledets multiplicative 
ergodic theorem (MET); 
here we use the two-sided version of the theorem for the cocycle $\T$, with 
the assumptions explained below.
Oseledets MET gives us a direct sum decomposition (the 
so-called Oseledets splitting) $T_u M = E^1(u)\oplus
\cdots\oplus E^{d}(u)$, $u \;\mu-$a.e., into subspaces
of different asymptotic exponential growth. The subspaces 
$E^i(u)$, assumed here to be one-dimensional, have the following properties:
\begin{enumerate}
\item covariance property: for each $i = 1,2,\cdots, d$, 
	$\T(u,1)E^i(u) = E^i(\varphi(u))$. 
\item \label{propertyLE} 
	exponential growth/decay: there exist real numbers 
		$\lambda_i$, $i=1,\cdots,d$ such that $v \in E^i(u) \neq 0 \in \mathbb{R}^d$ 
		implies that 
	$\lim_{n\to\pm\infty}\dfrac{1}{n}\log\norm{\T(u,n)v} = 
	\lambda_i$. 
\end{enumerate}
The asymptotic rates $\lambda_i$, which are called the Lyapunov exponents (LEs), are
assumed to be, for simplicity, nonzero and distinct, and indexed in decreasing order, i.e., $\lambda_1 > \lambda_2 > 
\cdots >\lambda_d$. Suppose $d_u$ is the number of positive LEs. The unstable subspace $\Eu$ is defined as 
$\Eu := \oplus_{i=1}^{d_u} E^i(u)$.
As noted 
in \cref{subsec:inefficiency}, a chaotic system by definition has $d_u > 0$, 
and a nontrivial unstable subspace, consisting of nonzero vectors 
at $\mu-$a.e. phase point. The stable subspace is defined 
as $\Es := \oplus_{i=d_u+1}^d E^i$.  

\subsection{Ruelle's formula split along Oseledets spaces}
The unit vectors along $E^i$ are denoted as $V^i$ and are
also called the covariant Lyapunov vectors (CLVs). Suppose 
$X$ in the CLV basis (span of the $V^i$s) can be expressed 
as $X(u) := \sum_{i=1}^d a^i(u) V^i(u)$. We write $\Xu$ and 
$\Xs$ to represent the decomposition of $X$ along $\Eu$ and 
$\Es$ respectively, i.e., $\Xu(u) = \sum_{i=1}^{d_u} a^i(u) V^i(u)$ 
and $\Xs(u) = \sum_{i=d_u+1}^d a^i(u) V^i(u)$. Then, we can 
rewrite Ruelle's formula (Eq.\ref{eqn:ruelleResponseFormula}) as,
\begin{align}
	\qoi = \sum_{n=0}^\infty 
\langle D(J \circ \varphi_{n}) \cdot X^{\rm s}, \mu\rangle +  
\sum_{n=0}^\infty 
	\langle D(J \circ \varphi_{n}) \cdot X^{\rm u}, \mu\rangle  
\label{eqn:splitRuelleResponseFormula}.
\end{align}
The first term on the right hand side of the split formula 
(Eq.\ref{eqn:splitRuelleResponseFormula}) will henceforth be 
referred to as the stable contribution 
(denoted using the subscript ``stable'') and the second 
term as the unstable contribution (denoted using the 
subscript ``unstable'') to the overall sensitivity. The motivation 
for the split is that the stable contribution can now be computed 
as if the system were not chaotic, using a stable tangent equation that 
is developed below.
\subsection{Derivation of the stable contribution}
\label{subsec:stableContribution}
The stable contribution can 
be written as
\begin{align}
	\label{eqn:stableContributionRestatement}
	\langle J, \partial_s\mu^s|_{s=0}\rangle_{\rm stable}  
	&=  \sum_{n =0}^\infty  \langle D(J \circ \varphi_{n})
	\cdot \Xs, \mu\rangle =   
	 \sum_{n =0}^\infty \langle (DJ)_n\cdot \T(\cdot,n)
	\Xs, \mu\rangle.
\end{align}
Now we develop a stable iterative procedure for the above expression, 
that satisfies our constraint of a Monte Carlo convergence, under 
the assumption of uniform hyperbolicity -- a simplifying assumption 
on the dynamics that gives uniform rates of decay of perturbations 
along $\Es$ and $\Eu$ forward and backward in time respectively. 
To wit, in a uniformly hyperbolic system, 
there exist constants $C, \lambda > 0$ such that 
$\norm{\T(u, n)\Xs(u)} \leq C e^{-\lambda n}\norm{\Xs(u)}$, for all 
$n \in \mathbb{Z}^+$ and for $u \;\;\mu$-a.e. Such a uniform decay also 
applies backward in time to perturbations along $\Eu$, i.e., $\norm{\T(u, -n)\Xu(u)} 
\leq C e^{-\lambda n}\norm{\Xu(u)}$, for all $n \in \mathbb{Z}^+$, 
and for $u \;\;\mu$-a.e., with the same constants $C, \lambda > 0$.
Under the assumption that $\varphi$ is uniformly hyperbolic on $M$,
and the assumption that $\norm{DJ}, \norm{\Xs} 
\in L^\infty(\mu)$, at $\mu$ almost every $u$,
$\abs{DJ(\varphi_n(u))\cdot \T(u,n)\Xs(u)} \leq C 
e^{-\lambda n} \norm{DJ}_\infty \norm{\Xs}_\infty$, where 
$\norm{f}_\infty := \inf\left\{
\alpha: \mu\{ u \in M : \abs{f(u)} > \alpha\} = 0\right\}$ 
for a scalar function $f:M \to \mathbb{R}$; when $V$ is a vector field,
$\norm{V}_\infty$ is defined similarly with $\abs{f(u)}$ replaced 
with $\norm{V(u)}$. In other words, 
the $L^\infty$-norm of the integrand in Eq. \ref{eqn:stableContributionRestatement} 
is exponentially decreasing with $n$. 
\subsection{Computation of the stable contribution}
\label{subsec:stableContributionComputation}
As a result of the exponentially decaying summation, 
truncation at a small number of terms provides a good 
approximation. We suggest the 
following method that uses a tangent equation, to compute 
the stable contribution in practice, since a tangent 
solver is usually available. We introduce the 
stable tangent equation, named so for using only the 
stable component of the perturbation but otherwise 
resembling a conventional tangent equation,
\begin{align}
	\notag
	\zetas_{i} &= D\varphi(u_{i-1}) \zetas_{i-1} + \Xs_{i}, \;\; i = 0,1,\cdots,N-1 \\
	\zetas_{-1} &= 0 \in \mathbb{R}^d
	\label{eqn:stableTangentEquation}. 
\end{align}
We can show that using the solutions of the above stable 
tangent equation, the stable contribution can be approximated as,
\begin{align}
	\label{eqn:stableContribution}
	{\qois}_{\rm stable} \approx \frac{1}{N}\sum_{n=0}^{N-1}DJ(u_n) \cdot \zetas_n. 
\end{align}
In \cref{prop:stableContribution}, we show, under the assumption of 
uniform hyperbolicity, that the error in the 
above approximation decays as ${\cal O}(1/\sqrt{N})$.
\subsection{The unstable contribution: an ansatz}
\label{sec:unstableContribution}
In this section, we derive a regularized expression for 
the unstable contribution defined in Eq. \ref{eqn:splitRuelleResponseFormula}. 
Denoting the components of $X$ along the $i$th CLV by the scalar 
field $a^i$, we can write $\Xu := \sum_{i=1}^{d_u} a^i V^i$. Thus,  
the unstable contribution from Eq. \ref{eqn:splitRuelleResponseFormula}
can be written as,
\begin{align}
	\label{eqn:furtherSplitRuelleformula}
		\qois_{\rm unstable}  = 
	\sum_{n=0}^\infty 
	\sum_{i=1}^{d_u} \langle D(J \circ \varphi_n) \: \cdot\: a^i V^i, \mu\rangle,
\end{align}
with the underlying assumption that the series in 
Eq. \ref{eqn:furtherSplitRuelleformula} converges for each $i \leq d_u$. 
We first informally motivate the ensuing derivation 
of a mollified expression for Eq. \ref{eqn:furtherSplitRuelleformula}.  
The integrand can be viewed as a linear functional $l_n:\Eu(u) \to \mathbb{R}$, 
evaluated at $V^i(u)$, and defined by 
 $l_n(V(u)) := a^i(u) (DJ)_n\cdot \T(u, n) V(u)$, $V(u) \in \Eu(u)$. Recall that although $l_n$ may be bounded for a finite 
$n$, the bound is an exponentially increasing function making the evaluation 
of $\langle l_n(V^i)\rangle$ computationally infeasible. 
In particular, at $\mu$-a.e. $u$, there exists an $N(u)$ such that 
$\abs{l_n(V(u))} \leq \norm{a^i}_\infty\norm{DJ}_\infty c\: e^{\lambda_1 n}\: \norm{V(u)}$, 
for all $n\geq N(u).$ On the other hand, due to the 
convergence of Ruelle's formula, the ensemble average $\langle 
l_n(V)\rangle$ declines asymptotically at least 
when $V = V^i, i \leq d_u$. Since we ultimately 
want to compute $\langle l_n(V^i)\rangle$ as opposed to the 
pointwise values of $l_n(V^i)$, we propose the following 
ansatz for the unstable contribution, for some $Y^{i,n} \in {\Eu}^*$, 
the dual of $\Eu$ -- 
\begin{align}
	\label{eqn:heuristic}
	\qois_{\rm unstable} = \sum_{n=0}^\infty 
	\sum_{i=1}^{d_u} 
	\langle V^i \cdot Y^{i,n}, \mu\rangle.
\end{align} 
In particular, we require that 
the vector field $Y^{i,n}$ is bounded for all $n$ and 
additionally such that $\abs{\sum_{i=1}^{d_u} 
\langle Y^{i,n}\cdot V^i\rangle}$ 
exponentially decreases with $n$. Then, if 
the central limit theorem holds 
for the integrand above, 
the heuristic expression in Eq. \ref{eqn:heuristic} leads to a desired 
Monte Carlo algorithm 
via ergodic averaging. Essentially, the ansatz chosen to 
satisfy our computational constraints suggests a vector field $Y^{i,n}$ 
that captures the overall sensitivity of $J_n$ to perturbations along $V^i$.
It is important that the pointwise values of $l_n(V^i)$ are not matched. 
The reason is, 
whenever for a finite $n$ depending on $u$, $l_n$ 
is a bounded linear functional on $\Eu(u)$, the uniqueness of 
$Y^{i,n}(u) \in {\Eu}^*(u)$ from Riesz 
representation theorem gives $Y^{i,n}(u) = a^i(u) DJ_n(u)$. As a result, 
$Y^{i,n}$ does not satisfy our requirements anymore and 
the original problem of large variances of ergodic averages has not been solved. 
Thus, we may require that 
$l_n(V^i(u))$ not
equal $V^i(u)\cdot Y^{i,n}(u)$ at each $u$ and only that $\langle
l_n(V^i)\rangle = \langle V^i\cdot Y^{i,n}\rangle,$ at each $n$.  

\subsubsection{Reformulation of the unstable contribution}
\label{sec:unstableContributionDerivation}
In this section, beginning with the original expression in 
Eq. \ref{eqn:splitRuelleResponseFormula}, we derive a new 
expression for the unstable contribution, holding Eq. \ref{eqn:heuristic} 
as a motivation. From Eq. \ref{eqn:splitRuelleResponseFormula},  
fixing an $i\leq d_u$, and isolating the
$n$-th summand,
\begin{align}
	\label{eqn:productRuleForDerivative}
	\langle a^i\: D(J\circ\varphi_n)\;\cdot\:V^i, \mu \rangle 
	= \langle D(a^i J \circ \varphi_n)\;\cdot\; V^i, \mu\rangle
	- \langle (J\circ \varphi_n) Da^i \cdot V^i, \mu \rangle. 
\end{align}
First, assuming each $a^i$ is differentiable along $E^i$, 
Eq. \ref{eqn:productRuleForDerivative} is valid. 
Moreover, the second term is a time correlation at time $n$, 
between the functions $J$ and $Da^i\cdot V^i$. If both 
these functions are assumed to be continuous, then the correlation 
between them decays exponentially in time \cite{chernov}. That is, 
the second term would 
approach its mean exponentially fast, for some $\gamma \in (0,1)$:
\begin{align}
	\label{eqn:correlationDecay}
	\abs{ \langle (J\circ \varphi_n) Da^i \cdot V^i, \mu \rangle
	- \langle J,\mu\rangle\langle\; Da^i \cdot V^i, \mu\rangle } \sim {\cal O}(\gamma^n).
\end{align}
In fact, if $X$ is assumed to be smooth, $\sum_{i=1}^{d_u} \langle Da^i\cdot 
V^i\rangle = 0$ (this result is proved in Theorem 3.1(b) of \cite{ruelle}).
It then follows that $\abs{\langle 
J\circ\varphi_n\;\sum_{i=1}^{d_u} Da^i\cdot V^i\rangle}$ 
exponentially decreases in $n$, and hence the vector field $Y^{i,n}_1 := -J\circ\varphi_n\; Da^i $ potentially forms a part of $Y^{i,n}$. 
The restatement in Eq. \ref{eqn:productRuleForDerivative}
therefore confines the problematic derivative to the first term in 
Eq. \ref{eqn:productRuleForDerivative}. 
We can now focus our attention 
on the first term to obtain the remainder of $Y^{i,n}$. Applying measure preservation of $\varphi$ on the first term, 
we obtain, for some $k \in \mathbb{N}$,
\begin{align}
	\label{eqn:measurePreservation}
	\langle D(a^i J \circ \varphi_{n})\;\cdot\; V^i, \mu\rangle
	= \langle (D(a^i_k\; J_{n+k}))_k \;
	\cdot \; V^i_k, \mu\rangle,
\end{align}
where we have adopted the succinct notation $J_k$ to denote the function 
$J\circ \varphi_k$; we neglect writing the subscript when $k=0$. The motivation for using measure preservation
forward in time becomes clear in the subsequent steps.  Some 
intuitive reasoning can be immediately made however: $Y^{i,n}$ is a vector 
field that captures the ensemble average of the 
directional derivative of $J_n := J \circ \varphi_n$, without matching 
the pointwise derivatives. 
As a next step, we
use the covariance of $V^i$ to express the integrand in Eq. \ref{eqn:measurePreservation} as 
a linear functional on $V^i(u)$ since we want to obtain part of 
$Y^{i,n}(u)$. Putting $k = 1$,
\begin{align}
	\label{eqn:covarianceManipulation}
	\langle D(a^i J_{n})\;\cdot\; V^i, \mu\rangle
	= \langle (D(a^i_1\; J_{n+1}))_1 \;
	\cdot \; \dfrac{\T(u,1) V^i}{z^i}, \mu \rangle 
	 = \langle D(a^i_1\; J_{n+1}) \;
	\cdot \; \dfrac{V^i}{z^i}, \mu\rangle,
\end{align}
where we have introduced $z^i(u) := \norm{\T(u,1) V^i(u)}$, and 
used the chain rule to go from the second expression to the third. 
Note that $\langle \log|z^i|, \mu \rangle = 
\lambda_i$, and in a uniformly hyperbolic system $\norm{z^i}_\infty > e^\lambda/C$. 
In order to take 
full advantage of the downscaling 
offered by the $z^i$, we again 
rewrite the integrand in the following 
way that is valid because $z^i$ is 
differentiable along 
$E^i$:
\begin{align}
	\label{eqn:zAbsorption}
	\langle D(a^i J_n)\;\cdot\; V^i, \mu\rangle
	 = \langle D(a^i_1\; J_{n+1}/z^i) \;
	\cdot \; V^i, \mu\rangle - \langle J_{n+1} a^i_1 D(1/z^i) \cdot V^i, \mu
	\rangle.
\end{align}
One advantage of rewriting is immediately clear -- 
an infinite sum of the second 
term over $n$, to obtain the unstable 
contribution, is well-posed. This is
because, similar to the second term in 
Eq. \ref{eqn:productRuleForDerivative}, 
this term is in the form of a time
correlation. The other advantage in 
Eq. \ref{eqn:zAbsorption} is that 
it can be used iteratively to evaluate the left hand side, making the integrand 
asymptotically smaller in norm. We now develop 
such an iterative procedure by first noticing that Eq. \ref{eqn:zAbsorption}
is valid for any bounded function $J$. Indeed since $\norm{(a^i_1 J_{n+1})/(a^i z^i)}_\infty 
\leq \norm{J}_\infty\norm{1/z^i}_\infty \leq C e^{-\lambda} \norm{J}_\infty$, 
$(a^i_1 J_{n+1})/(a^i z^i)$ is also a bounded function and thus 
can replace $J_n$ on the left hand side of Eq. \ref{eqn:zAbsorption}. Doing 
this replacement we obtain,
\begin{align}
	\label{eqn:iterationBase}
	\langle D(a^i_1 J_{n+1}/z^i)\;\cdot\; V^i, \mu\rangle
	 = \langle D(a^i_2\; J_{n+2}/(z^i_1 z^i)) \;
	\cdot \; V^i, \mu\rangle - \langle \dfrac{a^i_2 J_{n+2}}{z^i_1} \; D(1/z^i) \cdot V^i, \mu
	\rangle.
\end{align}
Note that $(z^i_1 z^i)(u) = \norm{
\T(u,2)V^i(u)}$; for notational 
convenience we introduce the scalar 
function $y^{i,k}(u) := \norm{\T(u,k) 
V^i(u)} = \Pi_{j=0}^{k-1} z^i_j(u), k \in \mathbb{N}$.
Now Eq. \ref{eqn:zAbsorption} can be used as a base 
for recursion by substituting for the first term on 
its right hand side using Eq. \ref{eqn:iterationBase}. 
Thus Eq. \ref{eqn:zAbsorption} becomes,
\begin{align}
	\label{eqn:iterationOne}
	\langle D(a^i J_n)\cdot V^i, \mu\rangle
	 = \langle D\left(\frac{a^i_2\; J_{n+2}}{y^{i,2}}\right) \:
	\cdot \: V^i, \mu\rangle - \sum_{k=1}^2 
	\langle \dfrac{a^i_k J_{n+k} z^i}{
		y^{i,k}} 
	\: D(1/z^i) \cdot V^i, \mu
	\rangle.
\end{align}
Now the recursion can 
be continued by obtaining an expression for the first term on the 
right hand side of Eq. \ref{eqn:iterationOne}, by using 
$(a^i_2 J_{n+2})/(a^i y^{i,2})$ in place of $J_n$ and so on. We obtain 
the following expression in the infinite limit of applying this 
recursion,
\begin{align}
	\label{eqn:iterationLimit}
	\langle D(a^i J_n)\;\cdot\; V^i, \mu\rangle
	 = \lim_{k\to\infty} \langle D(a^i_k\; J_{n+k}/y^{i,k}) \;
	\cdot \; V^i, \mu\rangle - \sum_{k=1}^\infty 
	\langle \dfrac{a^i_k J_{n+k} z^i}{y^{i,k}} \;
	D(1/z^i) \cdot V^i, \mu
	\rangle.
\end{align}
In lemma 
\ref{lem:limit}, we show that 
the limit in the first term in Eq. \ref{eqn:iterationLimit} is 0. 
In fact, the result that is proved 
is that for a sequence of bounded functions 
$f_n$ which goes to 0 pointwise
almost everywhere, the sequence 
of ensemble averages of 
directional derivatives along 
the unstable directions also converges 
to 0.  On applying measure preservation 
to each summand in the second term of 
Eq. \ref{eqn:iterationLimit}, we obtain a series of time 
correlations of a function $J$ with another 
bounded function, so that Eq. \ref{eqn:iterationLimit}
becomes,
\begin{align}
	\label{eqn:almostUnstableContribution}
	\langle D(a^i J_n)\;\cdot\; V^i, \mu\rangle
	 =  - \sum_{k=1}^\infty 
	\langle \dfrac{a^i J_n z^i_{-k}}{y^{i,k}_{-k}} \;
	D(1/z^i_{-k})_{-k} \cdot V^i_{-k}, \mu
	\rangle.
\end{align}
The second term in the equation above is a converging
series because the $L^\infty$ norms of the integrands 
are exponentially decreasing with $k$. More 
clearly, we have at $\mu$-a.e. $u$ that 
\begin{align}
	\abs{\dfrac{a^i J_n z^i_{-k}}{y^{i,k}_{-k}}
	(D(1/z^i_{-k}))_{-k}\cdot V^i_{-k}} 
	\leq \norm{a^i}_\infty\norm{J}_\infty 
	\norm{D(1/z^i)\cdot V^i}_\infty
	\dfrac{z^i_{-k}}{y^{i,k}_{-k}} \leq 
	C' e^{-\lambda (k-1)}, 
\end{align}
and hence $$\norm{\dfrac{a^i J_{n} z^i_{-k}}{y^{i,k}_{-k}} \;
	(D(1/z^i_{-k}))_{-k} \cdot V^i_{-k}}_\infty \leq 
	C' e^{-\lambda (k-1)}.$$
Thus by dominated convergence applied to the sequence
$g_j$, $j = 1, 2,\cdots$ of bounded functions,
$$ g_j := -\sum_{k = 1}^j 
\dfrac{z^i_{-k}}{y^{i,k}_{-k}} \;
	(D(1/z^i_{-k}))_{-k} \cdot V^i_{-k}, $$
we obtain that $g^i := \lim_{j\to\infty} g_j^i$ is also 
a bounded function and that 
Eq. \ref{eqn:iterationLimit}
becomes,
\begin{align}
	\label{eqn:unstableContributionFinalForm}
	\langle D(a^i J_n)\;\cdot\; V^i, \mu\rangle
	 =  - 	
	 \langle a^i J_n \Big(\sum_{k=1}^\infty 
	\dfrac{z^i_{-k}}{y^{i,k}_{-k}}\;
	(D(1/z^i_{-k}))_{-k} \cdot V^i_{-k}\Big), \mu
	\rangle = \langle a^i J_n \: g^i, \mu\rangle.
\end{align}
By assumption, the summation over $n$ of the 
left hand side of Eq. \ref{eqn:unstableContributionFinalForm} converges. This implies, since the series 
on the right hand side must also converge, that $\lim_{n\to \infty}
\langle a^i J_n g^i,\mu\rangle = 0$. On the other hand, this 
limit must be equal to $\langle J, \mu\rangle 
\langle a^i g^i, \mu\rangle$ since time correlations 
must decay to 0 on hyperbolic attractors. Thus, we have 
$\langle a^i g^i, \mu \rangle = 0$ since 
this is true for any bounded function $J$ that satisfies
the assumption that the series $\sum_{n=0}^\infty 
\langle D(a^i J_n)\cdot V^i, \mu \rangle$ converges. 
The derivation of 
Eq.\ref{eqn:unstableContributionFinalForm}
and showing that $g^i \in L^\infty$ complete the proof of \cref{thm:rued}.
Finally, note that the ansatz from section \ref{sec:unstableContribution}
is also valid. To see this, take 
$$ Y^{i,n}_2 = -\sum_{k=1}^\infty 	
	\dfrac{ a^i_{k} J_{n+k}  
 z^i}{y^{i,k}}\;
	D(1/z^i),$$
and set $Y^{i,n} = Y^{i,n}_1 + 	Y^{i,n}_2$. 
\subsection{Computation of the unstable contribution}
\label{subsec:unstableContributionComputation}
To complete the derivation of a regularized unstable contribution, we can rewrite the first 
term in Eq. \ref{eqn:productRuleForDerivative} by using the expression derived in Eq. \ref{eqn:unstableContributionFinalForm}. Thus, we obtain the following regularized unstable contribution,
\begin{align}
	\notag
	\qois_{\rm unstable} &= 
	\sum_{n=0}^\infty\Big( 
	\langle J_n \sum_{i=1}^{d_u} a^i g^i , \mu
	\rangle - 
	\sum_{i=1}^{d_u} \langle J_n Da^i \cdot V^i, \mu \rangle\Big) \\
	\label{eqn:unstableContributionFinal}
	&= \sum_{n=0}^\infty  
	\langle J_n \sum_{i=1}^{d_u} (a^i g^i  - Da^i \cdot V^i), \mu \rangle. 
\end{align}

In order to compute the unstable contribution in the form above, we resort 
to ergodic approximation of the ensemble average. Since we expect the time correlation between the bounded function 
$g := \sum_{i=1}^{d_u} a^i g^i - Da^i\cdot V^i$ and $J$ to decay exponentially 
in a uniformly hyperbolic system, the summation over 
$n$ would converge (to within machine precision of the true unstable 
contribution) with a small number of terms, when compared to $N$, the 
trajectory length used for an ergodic average approximation 
of each term. Thus, 
the computational time for the unstable contribution 
is roughly equal to that for evaluating $g$ along $N$ points. The function 
$g^i$ is naturally in the form of an iteration and thus 
can be obtained along a trajectory, by solving the following set of 
$d_u$ scalar equations, $1\leq i\leq d_u$, setting $\beta^i_{-1} = 0$,
\begin{align}
	\label{eqn:ruedEquation}
		\beta^i_{k+1} &= \beta^i_k/z^i_k + (D(1/z^i_k))_k\cdot V^i_k, \;\;
	k = 0,1,\cdots.
\end{align}
The solutions $\beta^i_{K}$,
$K \in \mathbb{Z}^+$, approximate the scalar function $g^i$, asymptotically. 
That is, for large $K$, $g^i(\varphi_K(u)) 
\approx -\beta^i_K$. Using this approximation of $g^i$ and a finite difference 
approximation of $Da^i\cdot V^i$, we can obtain the function $g$ along 
a primal trajectory starting from a $\mu$-typical phase point $u$. Then,
the numerical approximation of the unstable contribution is the following 
ergodic average,
\begin{align}
	\label{eqn:approximateUnstableContribution}
	\dfrac{d\langle J\rangle}{ds}_{\rm unstable} &\approx
	\dfrac{1}{N}\sum_{n=0}^M\sum_{i=0}^{N-1}J(u_{n+i}) g(u_i). 
\end{align}
Ignoring the numerical errors in the computation of $g$, \cref{lem:stableRuelleFormula} shows that the error in the approximation above decays as 
	${\cal O}(1/\sqrt{N})$. This completes the proof of \cref{thm:derivationOfS3}.

\section{Interpretation of the unstable contribution}
\label{sec:unstableContributionInterpretation}
In the previous section, we rewrote each term of Ruelle's 
formula, which represents the ensemble average of an 
unstable derivative, $\langle DJ_n\cdot \Xu\rangle$,
as a time correlation integral $\langle J_n 
\;g, \mu\rangle$ where $g$ was a bounded distribution 
that we obtained through an iterative procedure. 
In this section, we provide physical intuition for $g$ 
by relating it to the change in the SRB measure due 
to a perturbation along $\Xu$. 

We start with the simple case in which the SRB measure is 
absolutely continuous with respect to Lebesgue measure 
on the whole manifold $M$. For the derivation of 
an S3 formula that assumes the existence of a density 
on the whole manifold, see \cite{nisha-ctr}.
Examples of systems where this is 
true include expanding dynamics on compact attractors 
that have no stable submanifolds. In these cases, the 
volume element $d\mu = \rho\: du$, where $du = \abs{dx_1\cdots dx_d}$
is the standard volume element, 
for some smooth function $\rho:M \to \mathbb{R}^+$. Then,
integration by parts of each term of the unstable contribution 
according to Ruelle's formula, 
can be performed as follows,
\begin{align}
	\notag
	\langle D(J\circ \varphi_n) \cdot \Xu, \mu \rangle 
	&= \int_M {\rm div}( J \circ \varphi_n \: \Xu )   \; \rho\; du 
	- \int_M J \circ \varphi_n \; {\rm div} \Xu\; \rho\; du \\
	\label{eqn:integrationByPartsDensity}
	&= \int_{M} {\rm div}
	(\rho\: J \circ \varphi_n\: \Xu)\; du
	-  \int_M (J \circ \varphi_n) \Big(  
	\dfrac{D\rho}{\rho}\cdot \Xu + {\rm div}(\Xu)\Big)
	\; \rho \; du.
\end{align}
By Stokes theorem, the first term in Eq. \ref{eqn:integrationByPartsDensity}
is a boundary integral that gives the flux of the vector 
field $\Xu$ at the boundary of $M$, which is 0. Thus, in this case, 
the function $g \equiv - (D\rho \cdot \Xu/\rho + {\rm div}(\Xu))$. 
Hence, $g \rho = - {\rm div}(\rho \Xu)$. 
Roughly speaking, Eq. \ref{eqn:integrationByPartsDensity} captures 
the average of the function $J_n$ multiplied by the 
change in the probability distribution. 
Hence, this definition of $\rho g$  matches our intuition since 
locally, the perturbation $\Xu$ stretches the standard volume ($du$) 
by $g \rho = - {\rm div}(\rho \Xu)$. Moreover,
as derived in section \ref{sec:derivation}, it is easy to 
see that $\langle g, \mu \rangle = \int g \rho \; du = 0$. 

Now consider the more general case where the SRB measure 
is not absolutely continuous on the whole manifold. Although the
derivation of Eq. \ref{eqn:integrationByPartsDensity} is not valid, 
an interpretation of the unstable 
contribution can be made using a similar argument. First we choose a
measurable partition, say $\xi$, such that 
each partition element $\xi(u)$ that contains 
$u$ lies within the local unstable manifold at $u$.  
Since conditional measures of SRB measures along 
unstable manifolds are absolutely continuous, 
(see \cite{ly, vaughn}
for constructions of measurable partitions and disintegration of 
SRB measures), we can write the conditional measure of 
$\mu$ on $\xi(u)$ as $\rho_u(w) dw$ for some function $\rho_u$, 
where $dw = \abs{dx_1\cdots dx_{d_u}}$ is the standard Euclidean 
volume element in $d_u$ dimensions. In coordinates, at any $w \in \xi(u)$, 
$\Xu(w)$ can be written as $\Xu(w) = \sum_{k=1}^{d_u} 
v^k(w) \partial_{x_k}$, for some scalar functions $v^k$. 
Using such a disintegration 
of the SRB measure, each term of the unstable contribution will 
then have the following form, taking $a^i = 1$ for simplicity,
\begin{align}
\langle D(J \circ \varphi_n)\cdot \Xu, \mu \rangle 
	&= \int_{M/\xi} \int_{\xi(u)} \sum_{k=1}^{d_u} 
	v^k(w) \dfrac{\partial J \circ \varphi_n}{\partial x_k} 
	  \: \rho_u(w) \; dw\; d\hat{\mu}, 
\end{align}
where $\hat{\mu}$ is the factor measure defined as 
the pushforward of $\mu$ under the projection map 
$\pi:M \to \xi$, which maps a phase point $u$ to $\xi(u)$, hence: $\hat{\mu} 
= \mu \circ \pi^{-1}$ \cite{vaughn}. From this point, we can 
treat the $d_u$-dimensional inner integral analogously to the previous 
case of the expanding map. In particular, we can 
apply integration by parts to the inner integral, and analogous
to Eq. \ref{eqn:integrationByPartsDensity}, we obtain, 
\begin{align}
	\notag
\langle D(J \circ \varphi_n)\cdot \Xu, \mu \rangle 
	&=  - \int_{M/\xi} \int_{\xi(u)} 
		J\circ\varphi_n \:  
	\Big(\sum_{k=1}^{d_u} 
	v^k \dfrac{\partial\rho_u}{\partial x_k}\Big)  
	   \; dw\; d\hat{\mu} \\
	 &- \int_{M/\xi}  \int_{\xi(u)} 
	 J \circ \varphi_n  
 	\sum_{k=1}^{d_u} 
	\dfrac{\partial v^k }{\partial x_k} 
	  \: \rho_u(w) \; dw\; d\hat{\mu}. 
\end{align}
Here, again we obtain an integral representing a flux term 
on the boundaries of $\xi(u)$, the integral over $u$ of which 
is 0, due to cancellations \cite{ruelle1}. Then, comparing 
with Eq. \ref{eqn:unstableContributionFinal}, we can 
see that 
$$ g \rho_u \equiv -\sum_{k=1}^{d_u} 
(v^k  \dfrac{\partial\rho_u}{\partial x_k} +  	
\dfrac{\partial v^k }{\partial x_k} \rho_u) = 
-\sum_{k=1}^{d_u}  \dfrac{\partial v^k \rho_u}{\partial x_k},$$
which is again a divergence of $\rho_u \Xu$ on pieces of unstable manifolds.
While this provides an intuitive interpretation of $g$, 
it does not lead to a straightforward computation since 
the densities on the unstable manifolds, denoted $\rho_u$ above, 
are unknown. This justifies resorting to an iterative procedure 
that we did in section \ref{sec:derivation}, since the formula 
in \ref{eqn:unstableContributionFinal} only makes use 
of known quantities computed along trajectories. 
The other primary motive that Eq. \ref{eqn:unstableContributionFinal}
fulfills is that the algorithm 
must not involve discretization of the phase space, but remain 
a Monte Carlo method of computing integrals, which have convergence 
rates that are independent of the dimension of the phase space.

\section{Comments on S3 computation}
\label{sec:algorithm}
Revisiting the sketch of the proof in \cref{sec:main}, the first term 
of Eq. \ref{eqn:newFormula} appears as is from the split 
Ruelle's formula in Eq. \ref{eqn:splitRuelleResponseFormula}. 
Piecing together all the work carried out in \cref{sec:unstableContribution},
an exponentially converging, regularized expression for the 
unstable contribution, the second term of Eq. \ref{eqn:splitRuelleResponseFormula}, 
is crystallized into Eq.\ref{eqn:unstableContributionFinal}.
Putting these two contributions together in Eq. \ref{eqn:splitRuelleResponseFormula}
completes the proof of \cref{thm:derivationOfS3}. Moreover, the error 
in the ergodic approximation of Eq. \ref{eqn:unstableContributionFinal}
decays as ${\cal O}(1/\sqrt{N})$ using an $N$-term ergodic average: this 
follows from \cref{lem:stableRuelleFormula}. Thus, combining 
this result with \cref{prop:stableContribution}, 
the overall S3 formula has an error that decays as a typical 
Monte Carlo integration, as we sought.

We now briefly discuss a na\"ive implementation of the S3 formula, 
postponing an efficient algorithmic implementation (see \cite{nisha-aiaa}
for a superficial report in the case $d_u = 1$) to a future work. Using 
a generic initial condition $u$ sampled according to $\mu$ on 
the attractor, a primal trajectory of length $N$, chosen large enough 
for convergence of ergodic averages, is obtained from the solution 
of Eq. \ref{eqn:primal}. Along the primal trajectory, 
we use Ginelli {\emph et al.}'s 
algorithm (see \cite{ginelli} for 
the algorithm and \cite{florian} for more details and a new 
convergence proof) to obtain $V^i_n, 0\leq n\leq N-1$, $1\leq i\leq d_u$. 
Additionally, we also apply Ginelli {\emph et al.}'s 
algorithm to the adjoint cocycle (dual of $\T$) to obtain a set of adjoint 
CLVs, also normalized at each $u$, and
denoted as $W^i_n, 0\leq n\leq N-1$, $1\leq i\leq d_u$. Note that 
since $\Es(u) \perp {\Eu}^*(u)$, $\Xs\cdot W^i = (X - \Xu)\cdot W^i = 0$. 
This fact is used in order to obtain the stable and unstable components 
$\Xu_n$ and $\Xs_n$ along a trajectory. 

To realize the stable contribution in 
practice, the iterative equation referred to 
as the stable tangent equation (Eq. \ref{eqn:stableTangentEquation}) is 
used, as suggested in \cref{subsec:stableContributionComputation}. For the 
unstable contribution, Eq. \ref{eqn:unstableContributionFinal} is 
computed as an ergodic average. For the computation of each $g^i$, 
Eq. \ref{eqn:ruedEquation} is used as suggested in 
\cref{subsec:unstableContributionComputation}. In the numerical examples 
discussed below, we use both analytical expressions 
and approximate finite difference calculations to obtain $z^i$ and 
$D(1/z^i)\cdot V^i$, along trajectories. An algorithm for computation 
of derivatives of scalar functions along CLVs will be discussed in a 
future work, along with an adjoint (reverse-mode) algorithm for S3, in the 
interests of serving a high-dimensional parameter space.

Before we close this section, we comment on the uniform hyperbolicity 
assumption. Firstly, note that the assumption has been used to obtain 
the desired error convergence of both the stable and unstable contribution; 
the split of Ruelle's formula itself does not require uniform hyperbolicity.
In particular, in the stable contribution, we used the uniform rates of contraction 
of stable vectors, in \cref{prop:stableContribution}. In the unstable contribution 
derivation (i.e., in proving \cref{thm:rued}), and in fact in 
Ruelle's linear response formula itself, we use the existence of an 
SRB measure, which is guaranteed on a compact uniformly hyperbolic attractor. 
To obtain 
the error convergence of the unstable contribution, 
we used exponential decay of correlations and 
the CLT, which only hold on a hyperbolic attractor, for 
H\"older continuous functions of some positive H\"older exponent. 
(see \cite{liverani} and 
section 6 of \cite{young}). While the function $J$ is in $C^2$ and hence 
in a H\"older class, we have only shown boundedness of $g^i$, but 
assumed exponential decay of correlations with $J$. However, 
if the two functions satisfy the finite first moment condition 
of Chernov (see Corollary 1.7 of \cite{chernov}), the assumption 
of CLT and exponential decay of correlations would be valid. Moreover, 
besides these caveats, the assumption of uniform hyperbolicity itself 
could appear restrictive enough to affect the applicability of our results 
to high-dimensional dynamical systems encountered in practice. In this regard, 
it is worth mentioning that in a widely accepted hypothesis due to Gallovotti 
and Cohen (\cite{gallavotti}, see also \cite{ruelle-general} for more 
comments on this hypothesis), many fluid systems, and more generally, 
statistical mechanical models, behave as if they were uniformly hyperbolic.
Several recent studies also provide supportive evidence, wherein 
numerical methods that, strictly speaking, assume some hyperbolicity for 
their derivation and convergence, work well in high-dimensional real-life 
models (see \cite{chekroun} for an example from climate dynamics and 
\cite{angxiu-jfm} for a turbulent fluid flow simulation).
\section{Numerical examples}
\label{sec:results}
\subsection{Smale-Williams solenoid map}
\label{sec:solenoid}
The Smale-Williams solenoid map is a classic 
example of low-dimensional hyperbolic dynamics. 
It is a three-dimensional map given by
$\varphi^s(u) = [s_1 + \dfrac{r-s_1}{4} + \dfrac{\cos(\theta)}{2},
		2 \theta + \dfrac{s_2}{4}\sin(2\pi\theta),
		\dfrac{z}{4} + \dfrac{\sin\theta}{2}]^T$,
where $u := [r,\theta, z]^T$ in cylindrical coordinates. 
The attractor is a subset of the solid torus at the 
reference values of $s_1= 1.4$ and $s_2 = 0$. The probability 
distribution on the attractor is an SRB distribution 
\cite{ruelle,young} that has a density on the unstable manifolds. 
\begin{figure}
\includegraphics[height=7cm,width=\textwidth]{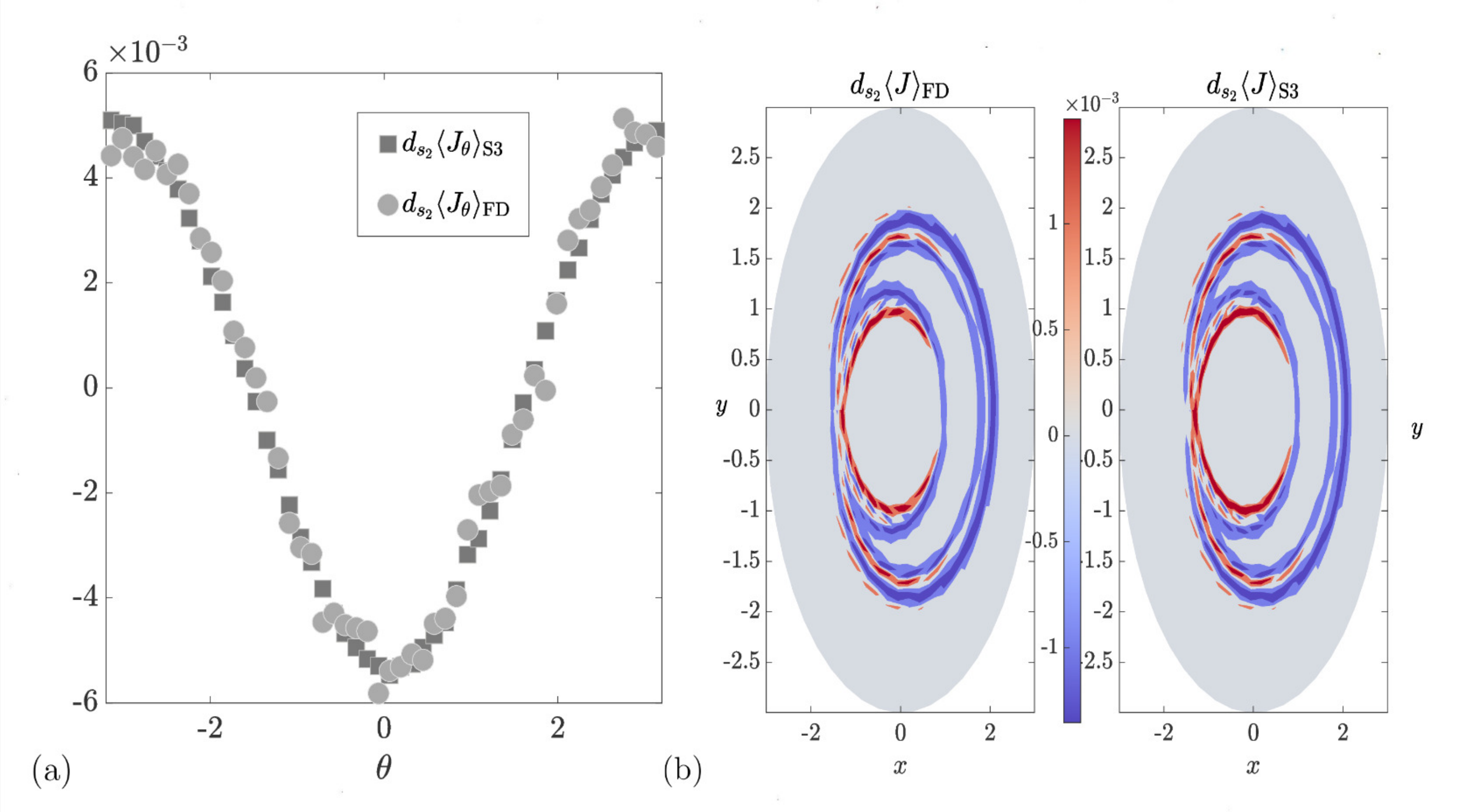}

\caption{Comparison of the sensitivities computed with S3 to 
	finite-difference for the solenoid map in Section \ref{sec:solenoid}.
	(a) $J$ is a set of 
	two-variable nodal basis functions along $r$ and $\theta$ axes.
	(b) $J_\theta$ is a set of nodal basis functions 
	along $\theta$ axis. }
\label{fig:solenoid}
\end{figure}
In this map, $r$ and $z$ directions form a basis for 
the stable subspace at each point (and the orthogonal
$\theta$ direction forms a basis for the adjoint unstable subspace).
Applying a perturbation to $s_1$ causes a stable
perturbation, i.e., the unstable contribution is zero,
since it affects only the $r$ coordinate. 
On the other hand, perturbing $s_2$ leads to a nonzero
unstable contribution.
A set of nodal basis functions along $r$ and $\theta$ 
is chosen to be the objective function. We use 
a na\"ive implementation of the S3 algorithm presented in 
\cref{sec:algorithm}. In order to validate the S3 computation, 
we compare the sensitivities $(d\langle J \rangle/ds_2)$
with finite-difference results generated using 10 billion
Monte Carlo samples on the attractor. The sensitivities
to the parameter $s_2$ are shown in Figure 1(a). 
In Figure 1(b), the objective function is a set of nodal basis functions
along the $\theta$ direction. From Figures 1(a,b),
we see close agreement between the sensitivities computed with (a more 
general version of) S3 and, finite-difference results, thus validating
both the stable and unstable parts of the S3 algorithm.
\subsection{Kuznetsov-Plykin map}
\begin{figure}
\includegraphics[height=6cm,width=\textwidth]{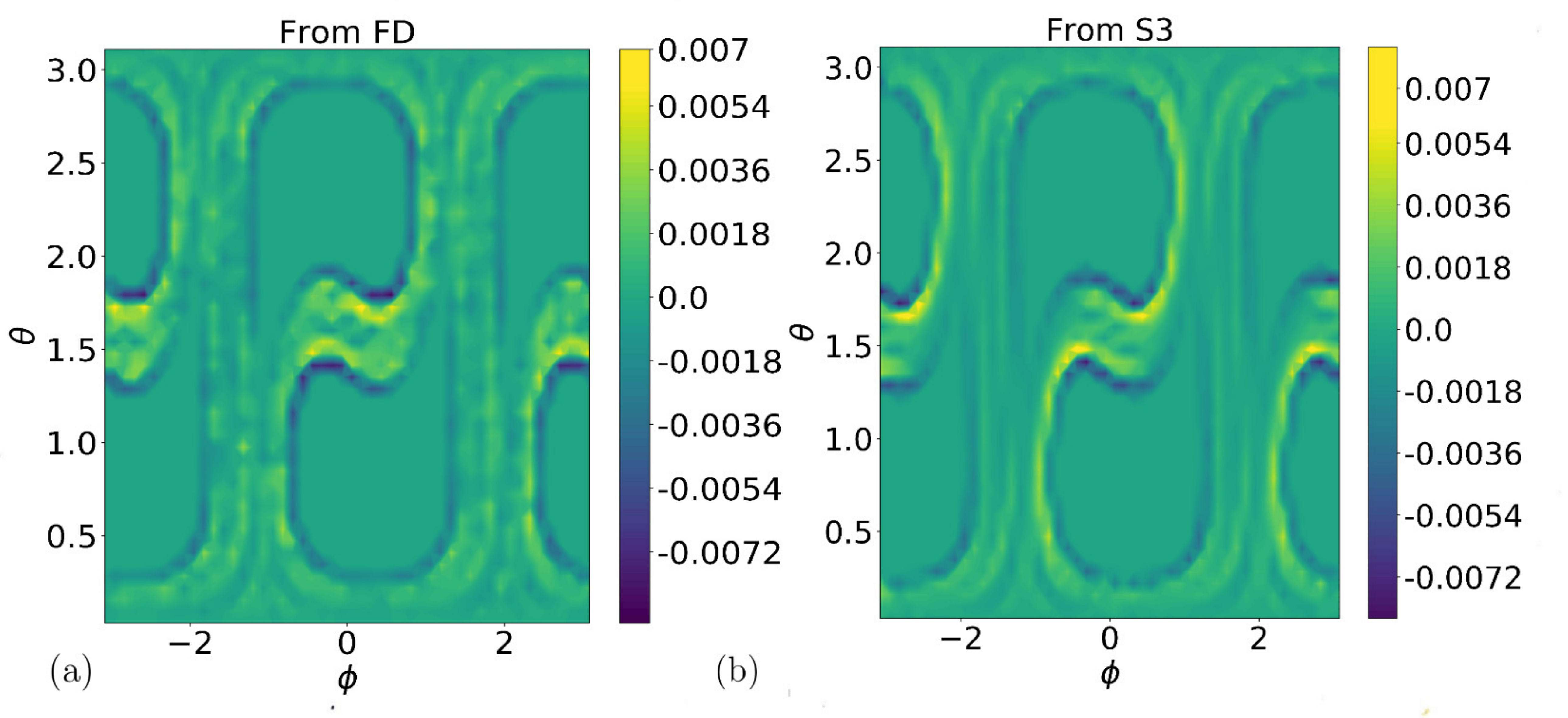}
\caption{Comparison of the sensitivities of 
	the nodal basis functions 
along the $\theta$ and $\phi$ axes to the parameter $s_2$
	obtained for the Kuznetsov-Plykin attractor using (a) finite difference and (b) the S3 algorithm.}
\label{fig:plykin}
\end{figure}
As a second test case for S3, we consider the Kuznetsov-Plykin 
map as defined by \cite{kuznetsov}, which describes
a sequence of rotations and translations on the 
surface of the three-dimensional unit sphere.
The two parameters we choose to vary are $s_1 := \epsilon$ 
and $s_2 := \mu$, 
which are defined by \cite{kuznetsov}. 
The map is given by $\varphi_{n+1}^s(u) = f_{-1,-1}\circ f_{1,1}(u)$
where $u = [x_1,x_2,x_3]^T \in \mathbb{R}^3$. For the function $f_{\cdot,\cdot}$ and further details regarding the 
hyperbolicity of the system, the reader is 
referred to \cite{kuznetsov}. The probability distribution 
on the attractor again is again of SRB type, with the 
existence of a density along the unstable manifolds.
We again use a na\"ive implementation of the S3 formula to compute the 
sensitivities as in the case of the solenoid map in 
Section \ref{sec:solenoid}.
The objective function $J$ is a set of nodal basis functions 
along the $\theta$ and $\phi$
spherical coordinate axes. The finite-difference sensitivities were 
computed with the
central difference around the reference value of $s_2 = 1$ 
by means of 10 billion independent
samples on the attractor. The results from S3 agree well with finite-difference sensitivities 
as shown in Figure \ref{fig:plykin}.

\section{Conclusions}
\label{sec:conclusions}
We have presented a tangent space-split sensitivity formula to compute
the derivatives of statistics to system parameters in chaotic dynamical systems. 
The algorithm to implement the formula 
requires the computation of a basis for the tangent and 
adjoint unstable subspaces along a long trajectory. The stable contribution to the overall sensitivity can be 
efficiently computed by a conventional tangent/adjoint computation just 
as in nonchaotic
systems. The unstable contribution has been rederived to be expressed as 
ergodic average that yields a Monte Carlo convergence. The numerical examples described in 
Section \ref{sec:results} satisfy the simplifying assumptions of 
uniform hyperbolicity that were made in the derivation.
They show close agreement with finite-difference results, serving as a 
proof-of-concept for the new formulation. In order to make the new formulation 
applicable to a high-dimensional problem, more work is needed toward an efficient 
implementation, particularly for the terms in Eq. \ref{eqn:ruedEquation}.

\section{Appendix}
\begin{proposition}
	\label{prop:stableContribution}
The error in an $N$-term ergodic approximation using the stable 
	tangent equation \ref{eqn:stableTangentEquation},
	of the stable contribution, decays as ${\cal O}(1/\sqrt{N})$. That is,
	$$e_N := \abs{\dfrac{d\langle J \rangle}{ds}_{\rm stable} - 
	\dfrac{1}{N} \sum_{n=0}^{N-1} DJ(u_n) \cdot\zetas_n} \leq \cs/\sqrt{N}, \cs > 0.$$
\end{proposition}
\begin{proof}
	It is easy to check that $\zetas_n := \sum_{i=0}^n \T(u_i, n-i) \Xs_i$ 
	satisfies Eq. \ref{eqn:stableTangentEquation}. So the 
	approximation to the stable contribution can be written, for some $M \leq N-1$ as, where $\sum_i^j = 0$ if $i > j$,
	\begin{align*}
		&\dfrac{1}{N}\sum_{n=0}^{N-1} \sum_{i=0}^n DJ(u_n)\cdot 
		\T(u_i, n-i) \Xs_i = \dfrac{1}{N}\sum_{n=0}^{N-1} 
		\sum_{i=n-M}^n DJ(u_n)\cdot 
		\T(u_i, n-i) \Xs_i \\
		&- \dfrac{1}{N}\sum_{n=0}^{M-1} 
		\sum_{i=n-M}^{-1} DJ(u_n)\cdot 
		\T(u_i, n-i) \Xs_i + 
		\dfrac{1}{N}\sum_{n=M+1}^{N-1} 
		\sum_{i=0}^{n-M-1} DJ(u_n)\cdot 
		\T(u_i, n-i) \Xs_i. 
	\end{align*}
	Thus,  
	\begin{align*}
		&e_N \leq
		\abs{ \dfrac{1}{N}\sum_{n=0}^{M-1} 
		\sum_{i=1}^{M-n} DJ(u_n)\cdot 
		\T(u_{-i}, n+i) \Xs_{-i}}  + 
		\abs{\dfrac{1}{N}\sum_{n=M+1}^{N-1} 
		\sum_{i=0}^{n-M-1} DJ(u_n)\cdot 
		\T(u_i, n-i) \Xs_i}\\
		&+\abs{ \dfrac{d\langle J \rangle}{ds}_{\rm stable} -  
		\dfrac{1}{N}\sum_{n=0}^{N-1} 
		\sum_{i=n-M}^n DJ(u_n)\cdot 
		\T(u_i, n-i) \Xs_i}.
	\end{align*}
	Under the uniform hyperbolicity assumption, we know that $\norm{\T(u,n) 
	\Xs(u)} \leq C e^{-\lambda n}$. Moreover, we assume that $\norm{DJ}, 
	\norm{\Xs}$ are bounded functions. Hence, where $\gamma:=
	\sum_{i=0}^\infty e^{-\lambda i}$,
	\begin{align}
		\notag
		e_N &\leq
		\dfrac{C \gamma^2 \norm{DJ}_\infty \norm{\Xs}_\infty}{N} 
		+ \dfrac{C' \gamma e^{-\lambda (M+1)} (N - (M+1)) 
		\norm{DJ}_\infty \norm{\Xs}_\infty}{N} \\
		\label{eqn:errorBound}
		&+\abs{ \dfrac{d\langle J \rangle}{ds}_{\rm stable} -  
		\dfrac{1}{N}\sum_{n=0}^{N-1} 
		\sum_{i=n-M}^n DJ(u_n)\cdot 
		\T(u_i, n-i) \Xs_i}.
	\end{align}
	To obtain an upper bound for the third term, again we 
	use the uniform hyperbolicity assumption. So,
	the integrand in the $n$th summand 
	of the stable contribution (Eq. \ref{eqn:splitRuelleResponseFormula}) 
	satisfies $\norm{(DJ)_n\cdot \T(\cdot, n)\Xs}_\infty \leq C 
	\norm{DJ}_\infty \norm{\Xs}_\infty e^{-\lambda n}$. Hence 
	$\norm{\sum_{n=0}^M (DJ)_n \cdot \T(\cdot, n)\Xs}_\infty \leq  C 
	\norm{DJ}_\infty \norm{\Xs}_\infty \gamma$, for any $M$ and 
	by dominated convergence, $\sum_{n=0}^\infty
	(DJ)_n \cdot \T(\cdot, n)\Xs \in L^1(\mu)$, 
	and the stable contribution can be written as, 
	\begin{align}
		\notag
		\dfrac{d\langle J \rangle}{ds}_{\rm stable} 
		&= \langle \sum_{i=0}^\infty
		(DJ)_i \cdot \T(\cdot, i)\Xs, \mu\rangle.
	\end{align}
	Thus the third term in Eq. \ref{eqn:errorBound} has the following 
	bound, 
	\begin{align}
		\notag
		&\abs{ \dfrac{d\langle J \rangle}{ds}_{\rm stable} -  
		\dfrac{1}{N}\sum_{n=0}^{N-1} 
		\sum_{i=n-M}^n DJ(u_n)\cdot 
		\T(u_i, n-i) \Xs_i} \\
		\notag
		&\leq \abs{\langle \sum_{i=0}^M
		(DJ)_i \cdot \T(\cdot, i)\Xs, \mu\rangle -
		\dfrac{1}{N}\sum_{n=0}^{N-1} 
		\sum_{i=n-M}^n DJ(u_n)\cdot 
		\T(u_i, n-i) \Xs_i} \\
		\label{eqn:secondLine}
		&+   C \gamma 
	\norm{DJ}_\infty \norm{\Xs}_\infty e^{-\lambda M},
	\end{align}
	where the second term on the right hand side of 
	Eq. \ref{eqn:secondLine} again uses uniform hyperbolicity. 
	Applying measure preservation in each of the integrals 
	in the first term of Eq. \ref{eqn:secondLine},
	\begin{align}
		\notag
		&\abs{ \dfrac{d\langle J \rangle}{ds}_{\rm stable} -  
		\dfrac{1}{N}\sum_{n=0}^{N-1} 
		\sum_{i=n-M}^n DJ(u_n)\cdot 
		\T(u_i, n-i) \Xs_i} \\
		\notag
		&\leq \abs{
			 \langle \sum_{i=0}^M
		DJ \cdot \T(\varphi_{-i}(\cdot), i)\Xs_{-i}, 
		\mu\rangle -
		\dfrac{1}{N}\sum_{n=0}^{N-1} 
		\sum_{i=n-M}^n DJ(u_n)\cdot 
		\T(u_i, n-i) \Xs_i} \\
		\label{eqn:thirdLine}
		&+   C \gamma 
	\norm{DJ}_\infty \norm{\Xs}_\infty e^{-\lambda M}.
	\end{align}
	The integrand in \ref{eqn:thirdLine} is continuous on $M$
	since $\T(u,n):\Es(u) \to \Es(u_n)$
	is a continuous map and, $DJ(u): T_u M \to \mathbb{R}$ and  
	$\Xs : M \to \Es$ are continuous by assumption.
	Then we expect that $\sum_{i=0}^N DJ\cdot \T(\varphi_{-i}, i) \Xs_{-i}$ obeys 
	the central limit theorem \cite{chernov}. 
	Using this in Eq. \ref{eqn:thirdLine}, Eq. \ref{eqn:errorBound}, 
	gives, letting $M \to N-1$,
	\begin{align}
		\label{eqn:errorStable}
		e_N &\leq
		\dfrac{C \gamma^2 \norm{DJ}_\infty \norm{\Xs}_\infty}{N} 
		+\dfrac{{\rm var}[\sum_{n=0}^N DJ_n \cdot \Xs]}
		{\sqrt{N}} +  C \gamma 
		\norm{DJ}_\infty \norm{\Xs}_\infty e^{-\lambda (N-1)} 
		\in {\cal O}(1/\sqrt{N}).
	\end{align}
	\end{proof}
\begin{lemma}
	\label{lem:limit}
	If the pointwise limit of a sequence of bounded functions 
	$\left\{f_n\right\}_{n=0}^\infty \subset L^1(\mu)$ 
	vanishes, i.e., $\lim_{n\to\infty} f_n(u) = 0, u\; 
	\mu-{\rm a.e.},$
	then the sequence $\langle Df_n\cdot V, \mu\rangle$ converges to 
	0, when $V$ is an unstable vector field differentiable along $\Eu$. 
\end{lemma}
\begin{proof}
	Let $\xi$ be a measurable partition of $M$ such that for $\mu$-a.e.
 	$u$, the element of the partition containing $u$, denoted $\xi(u)$,
	is contained in a local unstable manifold of $u$, i.e., 
	$\xi(u) \subset W^{u}(u)$. We assume a $\xi$ constructed according
	to Ledrappier-Young's Lemma 3.1.1 \cite{ly}, and 
	also use the Lyapunov-adapted coordinates introduced there. 
	In a neighborhood of each $u$, let 
	$\Phi_u:M \to [-\delta,\delta]^{d_u} \oplus [-\delta,\delta]^{d_s}$ be the adapted coordinate system such 
	that $\Eu(u), \Es(u)$ are identified with $\mathbb{R}^{d_u}, 
	\mathbb{R}^{d_s}$ respectively.
	Ledrappier-Young prove the existence of a measurable 
	function $\delta$ depending on $u$ in order for 
	$\xi$ to be a measurable partition of $M$; note however 
	that in our more specific case of a uniformly hyperbolic 
	compact attractor, we can choose a $\delta$ independently
	of $u$ (see 
	section 6.2 of \cite{katok}).
	
	Furthermore, the image of $W^u(u)$ under $\Phi_u$ is 
	a neighborhood of the 
	origin in $\mathbb{R}^{d_u}$, i.e., 
	the last $d_s$ coordinates of $\Phi_u(W^u(u))$ 
	are 0. Let the image of $\xi(u)$ under 
	this map be $B_u \subset [-\delta,\delta]^{d_u}$. 
	If $x_1,\cdots,x_{d_u}$ 
	are Euclidean coordinate functions in $\mathbb{R}^{d_u}$, the 
	pushforward of $V|_{\xi(u)}(w) \in \Eu(w)$ through $\Phi_u$
	can be expressed as $V(w) = \sum_{k=1}^{d_u} v_k(x)
	\partial_{x_k}$, $w \in \xi(u)$, and $x = \Phi_u(w)$, for 
 	differentiable functions $v_k:B_u\to \mathbb{R}$. 
Since $\xi(u)$ is a measurable partition, 
we can apply disintegration of $\mu$ on $\xi$, which 
gives for some measurable set $E$ that 
$\mu(E) = \int_{M/\xi} 
\int_{\xi(u)} \mathbbm{1}_E(w) \; d\mu_{\xi(u)}(w)
	d\hat{\mu}(\xi(u))$ \cite{vaughn, young}. Here the conditional measures 
of $\mu$ on $\xi(u)$ are denoted as $\mu_{\xi(u)}$ 
and the factor measure on the quotient space 
$M/\xi$ is denoted as $\hat{\mu}$. Given that 
$\mu$ is an SRB measure of $\varphi$, 
the conditional measure $\mu_{\xi(u)}$ 
is absolutely continuous with respect to 
${d_u}$-dimensional volume measure,
at $\mu$ almost every $u$; let the corresponding 
probability density 
function be denoted by $\rho_u: B_u \to \mathbb{R}^+$.
Using this setup, each term of the sequence of 
our interest is, where $dx = |dx_1\cdots dx_{d_u}|$ is the 
standard $d_u$-dimensional volume element, and 
$\tilde{f}_n := 
f_n \circ \Phi_u$,
\begin{align}
	\langle D f_n \cdot V, \mu\rangle = 
	\int_{M/\xi} 
	\int_{B_u}   
	\sum_{k=1}^{d_u} v_k(x)\;
	\dfrac{\partial \tilde{f}_n}{\partial x_k}
(x) \: 
	\rho_u(x)\; dx \; d\hat{\mu}.
\end{align}
Rewriting the integrand we obtain,
\begin{align}
	\notag
	\langle D f_n \cdot V, \mu\rangle = 
	&\int_{M/\xi} 
	\int_{B_u} \sum_{k=1}^{d_u} 
	\dfrac{\partial}{\partial x_k} (\tilde{f}_n \: v_k \:\rho_u) \:  dx\; 
	d\hat{\mu} \\ 
	&- \int_{M/\xi} 	
	\int_{B_u}   
	\tilde{f}_n  \; 
	\sum_{k=1}^{d_u}
	\dfrac{\partial (\rho_u v_k)}{\partial x_k} \; dx\; d\hat{\mu}.
\end{align}
The first term goes to zero at each $n$ due to cancellations 
along the boundaries of $B_u$ at different $u$s \cite{ruelle1}. To see this, 
choose a finite cover $\cup_{l\leq r}\xi(u_l) \supset M$ and 
take a partition of unity supported on each $B_{u_l} = \Phi_{u_l}(\xi(u_l))$ 
so that the integrals on the boundaries of $B_{u_l}$ are 0. Using dominated 
convergence, the second term converges to 0 as $n \to \infty$ 
at $\mu$ almost every $u$. Hence 
$\lim_{n\to \infty} 
\langle Df_n \cdot V, \mu\rangle = 0$.
\end{proof}

\begin{lemma}{\label{lem:stableRuelleFormula}}
The approximate formula Eq.\ref{eqn:approximateUnstableContribution} 
for the unstable contribution has an error that decays as ${\cal O}(1/\sqrt{N})$:
$e_{N,M} := \abs{\langle J, \partial_s\mu^s\rangle_{\rm unstable} -
	(1/N)\sum_{n=0}^M\sum_{i=0}^{N-1}J(u_{n+i}) g(u_i)} \leq \cu/\sqrt{N},$
	as $M \to N$,
	for some $\cu > 0$.
\end{lemma}
\begin{proof}
Suppose that the central limit theorem applies to the function 
$\sum_{n=0}^M J \circ \varphi_n\: g$. Then, we can say that,
for a sufficiently large $N$,
\begin{align}
	\label{eqn:cltforJg}
	\abs{\frac{1}{N}
	\sum_{i=0}^{N-1}\sum_{n=0}^M J(u_{n+i}) g(u_i) - 
	\langle \sum_{n=0}^M J\circ\varphi_n\: g, \mu\rangle} \leq 
	\dfrac{c_1}{\sqrt{N}}.
\end{align}
Further assuming that the decay of correlations between $J\circ \varphi_n$ 
	and $g$ is exponentially fast, we have, for every $n \in \mathbb{Z^+}$,
\begin{align}
\label{eqn:exponentialDecayOfCorrelations}
	\abs{\langle J\circ \varphi_n g, \mu \rangle - 
	\langle J\rangle \langle g\rangle} \leq c_2 \gamma^n, \;\;0\leq \gamma < 1.
\end{align}
	Since we have already shown in the main text (\cref{subsec:unstableContributionComputation}) that $\langle g\rangle = 0$, Eq. \ref{eqn:exponentialDecayOfCorrelations} gives $\abs{\langle J\circ \varphi_n g, \mu \rangle} \leq c_2 \gamma^n$. Thus, 
considering the sequence of functions $h_m := \sum_{n=0}^m J\circ \varphi_n\: g$,
\begin{align}
	\abs{\langle h_m , \mu\rangle} \leq 
	c_2 \sum_{n=0}^m \gamma^n \leq c_2/(1 - \gamma).
\end{align}
Thus, $\left\{h_m \right\} \in L^1(\mu)$ and $\norm{h_m}_1 \leq 
c_2/(1 - \gamma)$, and so by dominated covergence theorem, we have that,
	$\sum_{n=0}^\infty \langle J\circ\varphi_n\: g, \mu\rangle = \langle \sum_{n=0}^\infty
	J\circ\varphi_n \: g, \mu\rangle$. Then, again using Eq. \ref{eqn:exponentialDecayOfCorrelations}, we have
\begin{align}
	\label{eqn:remainder}
	\abs{\sum_{n=0}^\infty\langle J \circ \varphi_n \: g, \mu \rangle 
	- \langle \sum_{n=0}^M J\circ\varphi_n\: g, \mu \rangle} 
	\leq c_2 \gamma^M/(1 - \gamma).
\end{align}
Finally,
	\begin{align}
		e_{N,M} \leq \abs{\frac{1}{N}
	\sum_{i=0}^{N-1}\sum_{n=0}^M J(u_{n+i}) g(u_i) - 
	\langle \sum_{n=0}^M J\circ\varphi_n\: g, \mu\rangle} + 
	\abs{ \sum_{n=0}^\infty \langle J \circ \varphi_n \: g, \mu \rangle 
	- \langle \sum_{n=0}^M J\circ\varphi_n\: g, \mu \rangle},
	\end{align}
which using Eq. \ref{eqn:cltforJg} and Eq. \ref{eqn:remainder} gives,
	\begin{align}
		e_{N,M} \leq  c_2 \gamma^M/(1 - \gamma) + 
	\dfrac{c_1}{\sqrt{N}}.
	\end{align}
Taking the limit $M \to N$, this results in $e_{N, N} \in {\cal O}(1/\sqrt{N})$.
\end{proof}

\section*{Acknowledgments}
The authors are very grateful to Semyon Dyatlov for contributing to this 
work through many valuable discussions. 

\bibliographystyle{siamplain}
\bibliography{refs}

\end{document}